\documentclass[a4paper,12pt]{amsart}
\usepackage[utf8]{inputenc}
\usepackage{csquotes}
\usepackage[english]{babel}
\usepackage{cleveref}
\usepackage{amsmath, amssymb, amsthm}
\usepackage{mathrsfs}
\usepackage{mathtools}
\usepackage{tikz-cd}

\usepackage[backend=bibtex,style=alphabetic,giveninits,doi=false,url=false]{biblatex}
\DeclareFieldFormat[book,article,misc,incollection]{title}{\textit{#1\isdot}}
\DeclareFieldFormat[book,article,misc,incollection]{journaltitle}{#1\isdot}
\renewbibmacro{in:}{%
  \ifboolexpr{%
     test {\ifentrytype{article}}%
     or
     test {\ifentrytype{book}}%
     or
     test {\ifentrytype{misc}}%
  }{}{\printtext{\bibstring{in}\intitlepunct}}%
}
\DeclareFieldFormat{postnote}{#1}
\DeclareFieldFormat{multipostnote}{#1}

\makeatletter
\@namedef{subjclassname@2010}{%
  \textup{2010} Mathematics Subject Classification}
\makeatother

\newtheorem{theorem}{Theorem}[section]

\newtheorem{lemma}[theorem]{Lemma}
\newtheorem{corollary}[theorem]{Corollary}

\newtheorem{remark}[theorem]{Remark}

\providecommand{\BC}{{\mathbb{C}}}
\providecommand{\BP}{{\mathbb{P}}}
\providecommand{\BQ}{{\mathbb{Q}}}
\providecommand{\BR}{{\mathbb{R}}}
\providecommand{\BZ}{{\mathbb{Z}}}

\DeclareMathOperator{\rank}{rank}
\DeclareMathOperator{\Pic}{Pic}

\title[Mazur's Conjecture and An Unexpected Rational Curve]{Mazur's Conjecture and An Unexpected Rational Curve on Kummer Surfaces and their Superelliptic Generalisations}
\author{Damián Gvirtz}
\address{Department of Mathematics\\ South Kensington Campus\\ Imperial College London\\ LONDON\\ SW7 2AZ\\ United Kingdom}
\email{d.gvirtz15@imperial.ac.uk}
\date{}

\begin{document}

\baselineskip=17pt

\begin{abstract}
 We prove the following special case of Mazur's conjecture on the topology of rational points. Let $E$ be an elliptic curve over $\BQ$ with $j$-invariant $1728$. For a class of elliptic pencils which are quadratic twists of $E$ by quartic polynomials, the rational points on the projective line with positive rank fibres are dense in the real topology. This extends results obtained by Rohrlich and Kuwata-Wang for quadratic and cubic polynomials.
 
 For the proof, we investigate a highly singular rational curve on the Kummer surface $K$ associated to a product of two elliptic curves over $\BQ$, which previously appeared in publications by Mestre, Kuwata-Wang and Satgé. We produce this curve in a simpler manner by finding algebraic equations which give a direct proof of rationality. We find that the same equations give rise to rational curves on a class of more general surfaces extending the Kummer construction. This leads to further applications apart from Mazur's conjecture, for example the existence of rational points on simultaneous twists of superelliptic curves.
 
 Finally, we give a proof of Mazur's conjecture for the Kummer surface $K$ without any restrictions on the $j$-invariants of the two elliptic curves.
\end{abstract}



\maketitle
\section{Introduction}\label{sec:intro}
In the study of the distribution of rational points on varieties, two methods are frequently used to generate new points from existing ones: One can apply automorphisms defined over the ground field, e.g.\ arising from a group law on an elliptic curve. Or one can look for rational subvarieties that will be guaranteed to have many rational points. Often, a combination of both is needed. The prevalence of these methods is paramount to the whole subject.

A famous, successful example is Elkies' solution to Euler's conjecture on $A^4+B^4+C^4=D^4$ \cite{elkies}. In this paper, we consider another example given by Kuwata and Wang in \cite{kuwata}. Let $A$ be an abelian variety which is the product of two elliptic curves $E_1$ and $E_2$ over $\BQ$. Assume that $E_1$ and $E_2$ do not both have equal $j$-invariants $0$ or $1728$. Let 
\begin{align*}
 E_1&:y_1^2=x^3+ax+b=:g(x)\\
 E_2&:y_2^2=t^3+ct+d=:f(t)
\end{align*}
be affine equations for the elliptic curves in Weierstrass form (in particular $a=b=0$ and $c=d=0$ are excluded). The assumption on the $j$-invariant excludes exactly the cases $a=c=0$ and $b=d=0$. An affine model of the Kummer surface $K$ associated to $A$ is given after setting $y=y_1/y_2$:
\[
K:(t^3+ct+d)y^2=x^3+ax+b.
\]
On this surface, \cite[\S1]{kuwata} constructs a parametric curve $C$ as the scheme-theoretic image of the morphism
\[\sigma:\BP^1\to K,u\mapsto (x,y,t)(u):=\left(\frac{du^6-b}{a-cu^4},u^3,\frac{du^6-b}{u^2(a-cu^4)}\right).\]

Using this curve, one can prove the following theorem:
\begin{theorem}\cite[Theorem 3]{kuwata}\label{thm:mazur-wang}
 The set of rational points on $K$ is dense in the Zariski and real topologies.
\end{theorem}

This verifies, for the surface $K$, Mazur's conjecture on the topology of rational points: For any smooth variety $V$ over $\BQ$, if the rational points are Zariski dense in $V$, then their topological closure in the real locus $V(\BR)$ of $V$ is a union of real connected components of $V(\BR)$ \cite{mazur, mazur2}. It has been shown by a concrete counterexample \cite{counter} that Mazur's conjecture does not hold without further assumptions on the variety, although refined versions have been proposed that so far have resisted attempts at disproving them.

The same curve $C$ or rather its preimage $C'$ on $A$ was also independently found by Mestre in \cite{mestre} and used to prove that there are infinitely many elliptic curves over $\BQ$ of rank at least $2$ with a fixed $j$-invariant.

The appearance of $C$ is somewhat surprising and mysterious, given that the construction of $K$ starts with two generic elliptic curves and a priori there is not much reason to expect a rational curve over $\BQ$ on it apart from the obvious ones.

The discovery that prompted the present article is that the curve $C$ found by Mestre and Kuwata-Wang arises from a rather simple equation, which generalises to a wider class of surfaces. The precise statements and applications are contained in Sections \ref{sec:superelliptic}-\ref{sec:generalisation}, containing to the author's knowledge the first known case of Mazur's conjecture dealing with a class of quadratic twists of an elliptic curve by a quartic polynomial in \Cref{thm:mazur43}.

The last section does not utilise the curve $C$ and exhibits a proof of Mazur's conjecture for the Kummer surface $K$ without any assumptions on the $j$-invariants.

For the questions discussed in this article, it is not necessary to have projective models. We will thus mostly work with affine models that yield a dense open subvariety of the respective surface or curve. In our terminology, an affine, not necessarily geometrically irreducible curve has genus $0$ if it has a birational map to a projective curve whose desingularisation has genus $0$. A rational curve will always be an integral genus $0$ curve with a smooth rational point over the ground field.

After the publication of this article, the author was kindly informed by M. Ulas that the curve considered in \Cref{thm:Ckn} had previously been discovered by him \cite[Lem. 2.1]{ulas}.

\section{A Rational Curve on $K$ and Superelliptic Generalisations}\label{sec:superelliptic}

\begin{theorem}\label{thm:Ckn}
 Let $D_1$ and $D_2$ be two superelliptic curves over a field with arbitrary characteristic of the form
 \[D_1:y_1^k=x^n+ax+b\]
 \[D_2:y_2^k=t^n+ct+d,\]
 with $a$ or $b$ nonzero and $c$ or $d$ nonzero.
 The group $\mu_k$ of $k$-th roots of unity acts diagonally on $D_1\times D_2$. Let \[X=(D_1\times D_2)/\mu_k.\] An affine equation of $X$ is given by
 \[(t^n+ct+d)y^k=x^n+ax+b.\]
 Then there exists a genus $0$ curve $C$ on $X$ which is the closure of the subvariety of $X$ cut out by the affine equation
 \[(ct+d)y^k=ax+b.\]
 Moreover, if $a$ and $c$ are not both equal to $0$, $C$ has a rational component. If $k$ and $n$ are coprime, $b\neq 0$ and $a^nd^{n-1}-b^{n-1}c^n\neq 0$, then $C$ is geometrically irreducible. 
\end{theorem}
The condition $a^nd^{n-1}-b^{n-1}c^n\neq 0$ excludes the cases when there is an isomorphism between $D_1$ and $D_2$ that is compatible with the $\mu_k$-action.

For $k=2$ and $n=3$, this recovers Mestre's and Kuwata-Wang's curve on $K$. In this special case, these equations already appear in \cite{satge} (cf. \Cref{rem:satge} below).

\begin{proof}
We derive an alternative affine model of $C$ after which a brief analysis of the geometrically irreducible components yields the desired results. A transformation of the equations for $C$ gives
\begin{align*}
 \frac{ax+b}{ct+d}&=\left(\frac{x}{t}\right)^n,\\
 y^k&=\left(\frac{x}{t}\right)^n.
\end{align*}
Setting $r:=x/t$, the first equation is equivalent to
\[tr(cr^{n-1}-a)=b-dr^n\]
which defines a plane curve $\tilde C$ in the variables $(r,t)$. Note that this equation is linear in $t$. We distinguish three different cases:
\begin{enumerate}
 \item There exists no point $(r_0,t_0)\in \tilde C$ with $r_0(cr_0^{n-1}-a)=0$: In this case \[\pi:\BP^1\to\tilde C:r\mapsto(r,(b-dr^n)/(r(cr^{n-1}-a))\] defines a birational map, hence $\tilde C$ is a rational curve parametrised by $r$.
 \item There exist points $(r_0,t_0)\in \tilde C$ with $r_0(cr_0^{n-1}-a)=0$, and neither $a=c=0$ nor $b=d=0$:  If $r_0=0$, we must have $b=0$. If $cr_0^{n-1}-a=0$, we must have $r_0^nd=b$ and thus $a^nd^{n-1}-b^{n-1}c^n=0$. The map $\pi$ is from above is non-constant and yields a component of $\tilde C$ parametrised by $r$. However, additional components with $r=r_0$ appear, onto which $\pi$ does not map dominantly.
 \item $a=c=0$ or $b=d=0$: If $a=c=0$, then $y^k=r^n=b/d$ and thus $C$ decomposes into components with constant $y$ and $r$. If $b=d=0$, then $C$ has three components cut out by $r^{n-1}=a/c$, $x=y=0$ and $x=t=0$ respectively.
\end{enumerate}

From now on, assume that we are in one of the first two cases and let $\tilde C_1$ be the closure of the image of $\pi$ in $\tilde C$.  Since $C$ is obtained from $\tilde C$ by the affine equation $y^k=r^n$ and $r$ is locally constant outside $\tilde C_1$, we only have to consider $\tilde C_1$.

Let $p$ be the characteristic exponent of the ground field. Let $s$ be the $p$-primary part of the greatest common divisor of $k$ and $n$, so that $k=sp^ik'$ and $n=sp^in'$ where $(k',n')=1$. Then geometrically, $C$ decomposes into components
\[C_\zeta:(y^{k'}-\zeta r^{n'})^{p^i}=0\]
where $\zeta$ runs over all $s$-th roots of unity. For $\zeta=1$, we get a reduced, geometrically irreducible component
\[C_1:y^{k'}=r^{n'}\]
defined over the ground field since it is fixed by the Galois action. The curve $C_1$ is well-known to be rational and a parametrisation is given by $\theta \mapsto (r,y)(\theta):=(\theta^{k'},\theta^{n'})$. The other $C_\zeta$ are Galois twists of $C_1$ and so have genus $0$ too.

If $k$ and $n$ are coprime, i.\,e.\ $sp^i=1$, then $C$ coincides with the geometrically irreducible component $C_1$. 
\end{proof}

A direct computation gives:

\begin{theorem}\label{thm:param}
 A parametrisation of $C_1$ is given by:
 \[\sigma:\BP^1\to C_1,u\mapsto(x,y,t)(u)=\left(\frac{du^{kn}-b}{a-cu^{kn-k}},u^n,\frac{du^{kn}-b}{u^k(a-cu^{kn-k})}\right)\]
\end{theorem}
 
\subsection{Further Remarks}
\subsubsection{Geometric Considerations involving $\BP^1\times\BP^1$}\label{rem:satge}
The original example by Mestre has been studied by P.\ Satgé in \cite{satge}. There, he utilises the natural map from $K$ to $\BP^1\times\BP^1$ together with the Riemann--Hurwitz theorem to develop a combinatorical criterion for when the preimage of a rational curve on the latter surface yields a rational curve on the former. Amongst the low-degree examples he retrieves with the help of this criterion is the Mestre curve.

\subsubsection{Geometric Considerations involving $A$}
A new different approach which we mention for geometric insight is to first understand the preimage $C'$ on $A=E_1\times E_2$. In what follows, we show how to derive that $C$ has genus $0$ by such arguments in the case of two \emph{generic} elliptic curves, i.\,e. with distinct $j$-invariants and without complex multiplication.

Let $O_1$ and $O_2$ be the points at infinity of $E_1$ and $E_2$. The Néron-Severi group of $E_1\times E_2$ is given by
$\BZ h\oplus\BZ v$ where
\[h:=E_1\times \{O_2\},\ v:=\{O_1\}\times E_2.\]
By Bézout's theorem, the intersection numbers of $C'$ (cut out of $A$ by a quadric in each of the factors of an embedding $E_1\times E_2\subset \BP^2\times\BP^2$ by Weierstrass equations) with $h$ and $v$ are both $6$, therefore the class of $C'$ in the Néron--Severi group is $6h+6v$. Hence by the adjunction formula we deduce $p_a(C')= 37$. We can compute the singularities of $C'$: $(O_1,O_2)$ is a singularity with multiplicity $4$, \[\{(p_1,p_2)\in E_1[2]\setminus O_1\times E_2[2]\setminus O_2\}\] is a set of 9 singularities with multiplicity $2$ and $V(x=t=0)$ is a set of 4 singularities with multiplicity $3$. All singularities are ordinary and $C'$ does not pass through other torsion points than the ones mentioned. Hence $C'$ has geometric genus $10$.

We now use the Riemann-Hurwitz theorem. Before applying it to the double cover $C'\to C$, we first have to blow up the torsion points which are singular to get non-singular ramification points. If such a point $P$ has multiplicity $m_P$, then in the resolution we will have $m_P$ points of ramification index $2$. Indeed, after doing this, in the case $j(E_1)\neq j(E_2)$, Riemann-Hurwitz substitutes to
\[18= 2g(C')-2=2(g(C)-2)+\sum_{P\in\widetilde{C'}}(e_P-1)=2(g(C)-2)+(1\cdot4+9\cdot2)(2-1)\]
 and thus $g(C)=0$.

\subsubsection{Degenerate cases}
 In the cases of geometrically isomorphic $D_1$ and $D_2$ (i.e.\ $a^nd^{n-1}-b^{n-1}c^n=0$ and in particular $a=c=0$ or $b=d=0$), $C$ acquires geometric components which are the graphs of isomorphisms between $D_1$ and $D_2$.

\section{Simultaneous Twists of Superelliptic Curves}
As a corollary of \Cref{thm:Ckn}, we obtain similarly to \cite[Thm. 3]{kuwata}:
\begin{corollary}\label{cor:kn}
 Let $D_1$ and $D_2$ be superelliptic curves over a number field $L$ of the same form as in \Cref{thm:Ckn}. Assume that we are not in one of the cases $a=c=0$ or $b=d=0$. Then there exist infinitely many $[l]\in L^*/(L^*)^k$ such that the twists of $D_1$ and $D_2$ by $[l]$ both have an $L$-rational point.
\end{corollary}
By \emph{twist by $[l]$}, we mean in the case of $D_1$ the curve given by $ly_1^k=x^n+ax+b$ for a representative $l$ in the class $[l]$, and analogously for $D_2$. Up to $L$-isomorphism, it does not depend on the chosen representative. In the special case of $k=2,n=4$ and $k=n=3$, the theorem is a statement about genus $1$ models.

\begin{proof}
This proof follows the same idea as Kuwata-Wang but uses the newly found 
curve $C$ on $X$. Let $x(u),y(u),t(u)$ be as in \Cref{thm:param}. For a superelliptic curve $D$, denote by $D^l$ the twist by $l(L^*)^k$. Using $C$ gives us infinitely many points $(x,y,t)$ such that $D_1^{t^n+ct+d}$ and $D_2^{x^n+ax+b}$ have a rational point. Because $(t^n+ct+d)y^k=x^n+ax+b$, these are isomorphic to twists by the same class. We thus have a map
\[[\phi]:L^*\to L^*/(L^*)^k,u\mapsto (t(u)^n+ct(u)+d)(L^*)^k=(x(u)^n+ax(u)+b)(L^*)^k\]
such that $D_1^{[\phi](u)}$ and $D_2^{[\phi](u)}$ have a rational point.

Let $\phi(u):=t(u)^n+ct(u)+d$. It remains to show that $[\phi]$ does not have finite image. Suppose the image of $[\phi]$ is finite. Then there exists a finite set $S$ of places of $L$ such that $k\mid v(\phi(u))$ for all $u\in L, v\notin S$. This means by continuity that $v(\phi(u))\equiv0\bmod k$ for all $u\in L_v$. However, since $\phi\notin (L(u)^*)^k$ as a rational function (just by computing its numerator and denominator), there exists a point $P\in \BP^1_L$ such that $\phi$ has multiplicity $m$ prime to $k$ at $P$. Let $L(P)/L$ be the residue field extension of $P$. There are infinitely many places of $L$ that split completely in $L(P)$, so pick one $v\notin S$ amongst them and denote by $w$ an extension of $v$ to $L(P)$. Now $\phi$ has a zero or pole $P$ of multiplicity $m$ in $\BP_{L(P)_w}^1=\BP_{L_v}^1$ and in a neighbourhood of $P$, $v(\phi(u))$ cannot be divisible by $k$, yielding a contradiction.
\end{proof}

\section{Further Generalisations}\label{sec:generalisation}
The equation for $C$ in \Cref{thm:Ckn} gives rise to rational curves on an even wider class of surfaces where the exponents of $x$ and $t$ are chosen differently. Some of these curves have genus $0$ but do not contain a rational point. While the method of \Cref{thm:Ckn} does not apply to these generalisations, we can nevertheless give a few interesting examples and applications.

\subsection{Elliptic Curves with $j$-invariant $1728$}
Let $E$ be an elliptic curve with $j$-invariant $1728$ over a field $F$ of characteristic $\neq 2,3$. Let \[E:y^2=x^3+ax\] be an affine model of $E$ in Weierstrass form, in particular $a\neq 0$, and $f(t):=t^4+ct+d$ a polynomial with rational coefficients. Assume $c,d\neq 0$. Quadratic twisting by $f(t)$ yields an elliptic pencil $E^{f(t)}$. The situation at the degenerate fibres is irrelevant for our purposes.
\begin{theorem}\label{thm:C43}
 The surface over $F$ which is the total space of the pencil $E^{f(t)}$ contains a curve $C$ given by
 $(ct+d)y^2=ax$
 with an irreducible component $C_1$ given by $y=0$ and another rational irreducible component $C_2$.
\end{theorem}
\begin{proof}
 A parametrisation of $C_2$ is given by:
 \begin{align*}
  \sigma:&\BP^1\to C,\\
  &u\mapsto(x,y,t)(u)\\
  &=\left(\frac{(d^2/c^4)u^8-2dau^4+c^4a^2}{u^6},\frac{(-d/c^4)u^4+a}{u},\frac{(-d/c)u^4+c^3a}{u^4}\right)
 \end{align*}
 One can check explicitly that the image of $\sigma$ is indeed contained in $C$.
\end{proof}

In what follows we fix the parametrisation $\sigma$ above, which was found computationally using a uniformiser of the local ring at the smooth point $(X:Y:T:Z)=(0:1:0:0)$ on a projectivisation of $C$.

\begin{lemma}\label{lem:non-torsion}
 Assume $F=\BQ$. The set of $u\in\BQ$ such that $\sigma(u)$ has infinite order in its fibre $E^{f(t(u))}$ is dense in $\BR$.
\end{lemma}
\begin{proof}
    Define $E'_u:f(t(u))y^2=g(x)$, a family of elliptic curves parametrised by $u$. It has a section $\sigma'(u):=(x(u),y(u))$. After a finite base change $k(\sqrt{f(t(u))})/k(u)$, this family becomes trivial and $\sigma'$ is pulled back to the section $\sigma'':u\mapsto (x(u),y(u)\sqrt{f(t(u))})$. We infer that $\sigma''$ is not a torsion section since it intersects the identity section for $u=0$ but distinct torsion sections on elliptic surfaces have to be disjoint at smooth fibres (\cite[Rem. 11.3.8]{huybrechts} -- compare to the similar argument in \cite[VII.3.2]{miranda} for singular fibres). Hence, $\sigma'$ is not torsion either. Now the specialisation theorem (\cite[III.11.4]{silverman2}) says that for almost all $u$, $\sigma(u)$ is not torsion in its fibre.
\end{proof}

From this, one immediately deduces Zariski density of rational points:
\begin{corollary}\label{thm:zariski43}
 Assume $F=\BQ$. Infinitely many fibres of $E^{f(t)}$ have positive rank. More precisely, there is a set of $W\subset\BQ$, which is dense in the half-interval $(-d/c,\infty)$ if $ac>0$, respectively dense in $(-\infty,-d/c)$ if $ac<0$, such that the $E^{f(t)}$ has positive rank for all $t\in W$.
\end{corollary}
\begin{proof}
 The respective half-intervals given above are the images of $u\mapsto t(u)$. Now use \Cref{lem:non-torsion}.
\end{proof}
Note that density of the positive rank fibres in $E^{f(t)}$ over a non-empty open interval should be true if $E$ is any elliptic curve over $\BQ$ and $f$ is any polynomial with a real zero of odd order by a result of Rohrlich \cite[Thm. 2]{rohrlich}, conditional on the parity conjecture.

We deduce a new special case of Mazur's conjecture applied to elliptic pencils \cite[Conj. 4]{mazur}.
\begin{theorem}\label{thm:mazur43}
  Let $E$ be an elliptic curve over $\BQ$ with $j$-invariant $1728$ and let $f(t)=t^4+ct+d$ be a quartic polynomial over $\BQ$. Assume that $c,d\neq 0$ and $f(t)$ is non-negative for all $t\in\BR$. Then the set of $t\in\BQ$ with $\rank E^{f(t)}>0$ is dense in $\BR$.
\end{theorem}
A result by Rohrlich \cite[Thm. 3]{rohrlich} settled the case of $f$ being a quadratic polynomial using similar ideas as \cite{kuwata} for cubic polynomials. \Cref{thm:mazur43} complements Kuwata and Wang's quartic example $(t^4+1)y^2=x^3-4x$ \cite[p. 121]{kuwata} which they derived from the work by Elkies mentioned in the introduction. A recent preprint by Huang \cite{huang} deals with $d(t^4+1)y^2=x^3-x$ for some $d$. By entirely different methods and under some additional assumption, \cite[Prop. 1.1]{harpaz} proves Mazur's conjecture for the Kummer quotient associated to the product of non-trivial $2$-coverings of elliptic curves.

\begin{proof}
 View $E^{f(t)}$ as a genus $1$ pencil $E_x$ with respect to projection to $x$. A priori, the fibres do not have rational points but there are infinitely many which do. Namely, $O_u:=(x(u),y(u),t(u))$ and $O_{-u}:=(x(u),y(-u),t(u))$ are two (generically distinct) rational points in their fibre $E_{x(u)}$.
 
 Now by the same argument as in \Cref{lem:non-torsion}, for some choice of $u_0\in\BQ$ the point $(x_0,y_0,t_0):=O_{-u_0}$ has infinite order in $E_{x_0}$ with respect to the identity chosen as $O_{u_0}\in E_{x_0}$, as well as infinite order in $E^{f(t_0)}$. Using the group law on $E_{x_0}$, we spread $O_{-u_0}$ to get a dense set $T$ in a connected component of $E_{x_0}(\BR)$. By Mazur's torsion bound \cite{torsion} the rational points $(x,y,t)$ that are torsion in their fibre $E^{f(t)}$ lie in a proper Zariski-closed subset $S$ of the total space. The intersection $E^{f(t_0)}\cap S$ is finite because otherwise, one would have $E^{f(t_0)}\subset S$ but $O_{-u_0}\in E^{f(t_0)}\setminus S$. It follows that $T':=T\setminus S$ is dense in a connected component of $E_{x_0}(\BR)$. But by assumption on $f$, connected components of $E_{x_0}(\BR)$ project surjectively to $t$ so that the image of $T'$ projects densely to $t$.
\end{proof}
 
\subsection{Elliptic Curves with $j$-invariant $0$}
Let $E$ be an elliptic curve with j-invariant $0$ over a field $F$ of characteristic $\neq 2,3$. Let \[E:y^2=x^3+b\] be an affine model of $E$ in Weierstrass form and $f(t):=t^6+ct+d$ a polynomial with rational coefficients. Assume $b,c\neq 0$. Quadratic twisting by $f(t)$ yields an elliptic pencil $E^{f(t)}$. Once again, the situation at the degenerate fibres is irrelevant for our purposes.
\begin{theorem}\label{thm:C63}
 The surface which is the total space of the pencil $E^{f(t)}$ contains a curve given by
 $(ct+d)y^2=b$
 with a rational irreducible component.
\end{theorem}
\begin{proof}
A parametrisation is given by:
\begin{align*}
\sigma:&\BP^1\to C,\\
&u\mapsto(x,y,t)(u)\\
&=\left(\left(\frac{d^2}{b^2c^2}u^{12}-\frac{2db^5}{c^2}u^6+\frac{b^{12}}{c^2}\right)/u^{10},u^3/b^3,\left(\frac{b^7}{c}-\frac{d}{c}u^6\right)/u^6\right).
\end{align*}
One can check explicitly that the image of $\sigma$ is indeed contained in $C$.
\end{proof}

In what follows we fix the parametrisation $\sigma$ above, which was found computationally using a uniformiser of the local ring at the smooth point $(X:Y:T:Z)=(0:1:0:0)$ on a projectivisation of $C$. We can then prove an analogue to \Cref{thm:zariski43}.
\begin{lemma}
 Assume $d\neq 0$ and $F=\BQ$. Then infinitely many fibres of $E^{f(t)}$ have positive rank. More precisely, there is a set $W\subset\BQ$ which is dense in the half-interval $(-d/c,\infty)$ if $ac>0$, respectively dense in $(-\infty,-d/c)$ if $ac<0$, such that $E^{f(t)}$ has positive rank for all $t\in W$.
\end{lemma}
\begin{proof}
  By clearing denominators, the coefficients $a,b$ and $d$ can be assumed integral. We want to show that $\sigma(u)$ is non-torsion for a dense set of $u$. Set $u:=k/l$ with coprime $k,l\in\BZ$ and $s:=ck^6$. Then an integral model of $E^{f(t(u))}$ is given by:
  \[y'^2=x'^3+s^{24}f(t(u))d\]
  where $y'=s^{12}f(t(u))^2y$ and $x=s^8f(t(u))x$. For $l$ large enough $y'(u)=s^{12}f(t(u))^2y(u)$ is not integral and thus by the Lutz-Nagell criterion \cite[VIII.7.2]{silverman}, $\sigma(u)$ cannot be a torsion point.
  
  The respective half-intervals given above are the images of $u\mapsto t(u)$.
\end{proof}

\section{Proof of Mazur's Conjecture for the Kummer Surface of a Product Abelian Surface}\label{sec:mazurK}
In \cite[Thm. 3']{kuwata}, a sketch was given that extends \Cref{thm:mazur-wang} to a proof of Mazur's conjecture for all $j$-invariants. It has to proceed along lines different from \Cref{thm:mazur-wang} because the parametric curve is not available in the cases of equal $j$-invariants $0$ or $1728$. The strategy was to rely on the two elliptic pencils given by projections to $x$ and $t$ to spread rational points using the group laws. As communicated between the author and M.\ Kuwata, it is not clear whether this method is sufficient to get density in the real locus. We thus give a first proof.

\begin{theorem}\label{thm:mazur33}
 Let $K$ be the Kummer surface associated to the product of two arbitrary elliptic curves $E_1$ and $E_2$ over $\BQ$. Assume the rational points are Zariski dense in $K$. Then they are dense in the real topology of $K$.
\end{theorem}
\begin{proof}
 Recall that an affine equation of $K$ was given in \Cref{sec:intro} by
 \[f(t)y^2=g(x).\]
 Let $K_t$ and $K_y$ be the fibrations given by projections to the respective coordinates. Note that only the first comes equipped with a section and thus a natural group law. The fibres of $K_y$ are cubic curves and may not have a rational point.
 
 Let $t_1\in\BR$ be arbitrary. If we show that for any $\epsilon>0$, there exists an approximating $t'\in\BQ$ with $|t'-t_1|<\epsilon$ such that the topological closure $\overline{K_{t'}(\BQ)}$ is $K_{t'}(\BR)$, then we are done.
 
 Let $S$ be the Zariski closure of the set of rational points on $K$ that are torsion in their fibre $K_t$ or torsion in their fibre $K_y$ with respect to any of the inflection points chosen as identity. (The latter does not depend on the chosen inflection point since $3[I_1]=3[I_2]$ in $\Pic(K_y)$ for any inflection points $I_1, I_2 \in K_y(\BC)$ where $[\cdot]$ denotes the class of a divisor modulo linear equivalence.)
 
 \subsection*{Claim: $S\neq K$.} Assume $Q=(x,y,t)\in S(\BQ)$ is torsion in its fibre $K_t$. Then by Mazur's torsion bound, $Q$ lies in a proper closed subset $S_1$ of $K$.
 
 Now assume $Q=(x,y,t)\in S(\BQ)$ is torsion in its fibre $K_y$ with respect to some inflection point $I\in K_y(\BC)$. Then by Merel's torsion bound \cite{merel} for the number field $k(I)$, there is a bound $N$ (only depending on the uniformly bounded degree of the residue field $k(I)/\BQ$) such that $n_QQ=I$ for some positive $n_Q<N$. This can again be expressed by some necessary algebraic relations so that $Q$ lies in a proper closed subset $S_2$ of $K$. This proves the claim since $S\subset S_1 \cup S_2$.
 
 By assumption of Zariski density, there exists a point $P=(x_0,y_0,t_0)\in K(\BQ)$ outside of $S$. Because $S$ is algebraic, we know that $K_{y_0}\cap S$ is finite. Otherwise, one would have $K_{y_0}\subset S$ which is impossible since $P\in K_{y_0}\setminus S$. In the same way, we conclude that $K_{t_0}\cap S$ is finite.
 
 Multiples of $P$ with respect to the group law on $K_{t_0}$ are dense in the identity component of $K_{t_0}(\BR)$, which maps surjectively to the $y$-coordinate. Therefore we can replace $P$ without loss of generality by one such multiple $(x_0,y_0,t_0)$ which is not in $S$ with arbitrarily small $|y_0|$. Using this we may make two assumptions about $P$:
 
 \begin{enumerate}
  \item \label{it:connected} We can assume $|y_0|$ is sufficiently small such that $K_{y_0}(\BR)$ is connected. To see this, after setting $u:=\sqrt[3]{y_0}\in\BR$ and $\tau:=tu^2$, we can write $K_{y_0}(\BR)$ as:
 \[\tau^3+c\tau u^4+du^6=g(x).\]
 This is a family of curves parameterised by $u$ which is smooth in a neighbourhood of $u=0$. By Ehresmann's lemma, for small $|u|$ (and hence small $|y_0|$) its fibre is homeomorphic to the real curve $\tau^3=g(x)$, which in turn is homeomorphic to the connected real curve $v=g(x)$, where we set $v:=\tau^3$.
 \item \label{it:bound} Moreover, if $g(x)$ has three real roots, we define $m<0$ and $M>0$ as local minimum and maximum of $g(x)$ and assume that $|y_0|$ is sufficiently small such that
  \[m<f(t_1)y_0^2<M,\]
  where $t_1\in \BR$ is as in the beginning of the proof.
 \end{enumerate}
 
 Choose some inflection point $I_0\in K_{y_0}(\BC)$ as identity for the group law on $K_{y_0}$. Then by \Cref{lem:inflection} below, \[T:=\{(3n+1)P|n\in\BZ\}\subset K_{y_0}(\BQ).\] 
 
 By Assumption (\ref{it:connected}), $K_{y_0}(\BR)$ is isomorphic to the real Lie group $\BR/\BZ$ and $T$ is dense in it since $P$ is not torsion in $K_{y_0}$. Let $T':=T\setminus S$. Because $(T\cap S)\subset (K_{y_0}\cap S)$ is finite, the set of rational points $T'$ is also dense in $K_{y_0}(\BR)$. 
 
 We distinguish two cases to finish the proof of the theorem:

  \textbf{$g(x)$ has only one real root:} Then $K_t(\BR)$ is connected for all $t\in\BR$. We have to find a non-torsion point in $K_{t'}(\BQ)$ for some $t'\in\BQ$ with $|t'-t_1|<\epsilon$. The set $T'$ is dense in $K_{y_0}(\BR)$ and the projection from $K_{y_0}(\BR)$ to the $t$-coordinate is surjective. Hence the image of $T'$ under this projection is dense in $\BR$ and we can find $(x',y_0,t')\in T'$ with $|t'-t_1|<\epsilon$.
  
  \textbf{$g(x)$ has three real roots:} Then $K_t(\BR)$ has two connected components for all $t\in\BR$ and we denote its non-identity component by $N_t(\BR)$. It remains to show the existence of a rational point $P'\in N_{t'}(\BR)$ of infinite order in $K_{t'}$ for some $t'\in\BQ$ with $|t'-t_1|<\epsilon$.
  
  Observe that $K_{y_0}(\BR)\cap K_{t_1}(\BR)$ is the intersection of the elliptic curve $K_{t_1}(\BR)$ with the line $\{y=y_0\}$. By Assumption (\ref{it:bound}), this intersection consists of three points, of which exactly two lie in the \emph{oval} component $N_{t_1}(\BR)$. As $K_{y_0}$ is connected and $T'$ dense in $K_{y_0}(\BR)$, for any of these two intersection points $(x,y_0,t_1)\in N_{t_1}(\BR)$ we can find $P'=(x',y_0,t')\in T'$ such that $|t'-t_1|<\epsilon$ and $P'\in N_{t'}(\BR)$. 
\end{proof}

In classical geometric terms, the following lemma spreads rational points using secants and tangents without the need of a group law defined over the ground field.
\begin{lemma}\label{lem:inflection}
 Let $E$ be a plane cubic curve over a field $F$ and let $P\in E(F)$. Let $F'/F$ be a finite field extension and let $I\in E(F')$ be an inflection point. Equip $E_{F'}$ with the group structure with $I$ as neutral point. Then for all $n\in\BZ$, the multiple $(3n+1)P$ is $F$-rational.
\end{lemma}
\begin{proof}
 Denoting by $H\in\Pic(E)$ the class of a hyperplane section and by $[\cdot]$ the class of a divisor modulo linear equivalence, we have that \[D:=(3n+1)[P]-nH\] has degree $1$, so there exists a point $Q\in E(F)$ with $[Q]=D$. Then:
 \[(3n+1)([I]-[P])=[I]+nH-(3n+1)[P]=[I]-[Q].\]
\end{proof}

\begin{remark}
Relating the proof in the last section to the rest of the article, it should be mentioned that there is no possibility of applying the method of using several elliptic fibrations to cases beyond K3. Only K3 and abelian surfaces can contain distinct elliptic fibrations with sections \cite[Lem. 12.18]{shioda}. In particular, the case of quintic $f$ is out of reach.
\end{remark}

\subsection*{Acknowledgements}
The author thanks A.~Skorobogatov, M.~Kuwata, P.~Satgé, M.~Ulas and the anonymous referee who suggested how to prove \Cref{cor:kn} without using Faltings' theorem. This work was supported by the Engineering and Physical Sciences Research Council [EP/ L015234/1], the EPSRC Centre for Doctoral Training in Geometry and Number Theory (The London School of Geometry and Number Theory), University College London. 

\printbibliography
\end{document}